\documentclass[12pt]{article}

\usepackage{times}

\usepackage[a4paper,text={128mm,185mm}, centering]{geometry}

\usepackage{titlesec}

\titleformat{\section}[hang]%
{\bfseries\large}{\thesection.}{1ex}{}%

\titleformat{\subsection}[hang]%
{\bfseries}{\thesubsection}{1ex}{}%

\usepackage{amsthm}

\theoremstyle{plain}
\newtheorem{theorem}{Theorem}[section]

\newtheorem{proposition}[theorem]{Proposition}

\newtheorem{definition}[theorem]{Definition} 

\theoremstyle{definition}

\newtheorem{remark}[theorem]{Remark}

\newcommand{\pullbackmark}[2]{\save ;p+<.8pc,0pc>:(0,-1)::%
(#1) *{\phantom{Z}} %
;p+(#2)-(0,0) **@{-}%
;p-(#1)+(0,0) *{\phantom{Z}} **@{-} \restore}

\usepackage[matrix,arrow,curve,cmtip]{xy}
\usepackage{amssymb}
\usepackage{latexsym}
\usepackage{graphicx}

\begin{document}

\title{Span equivalence between algebras for $n$-globular operads}
\author{Yuya Nishimura}
\maketitle



\begin{minipage}{118mm}{\small
{\bf Abstract.} We define a new equivalence between algebras for $n$-globular operads which is suggested in [Cottrell 2015], and show that it is a generalization of ordinary equivalence between categories.\\
{\bf Keywords.} Algebras for $n$-globular operads, Span equivalence.\\
{\bf Mathematics Subject Classification (2010).} 18A22, 18C20, 18D50
}\end{minipage}


\section{Introduction}
In [Cottrell 2015], Thomas Cottrell  defined an equivalence of $K$-algebras on an $n$-globular set to show the following coherence theorem:
\begin{theorem}
Let $K$ be an $n$-globular operad with unbiased contraction $\gamma$, and let $X$ be $n$-globular set.
Then the free $K$-algebra on $X$ is equivalent to the free strict $n$-category on $X$.
\end{theorem}
\noindent
His equivalence in this theorem is as follows:
\begin{definition}
Let $K$ be an $n$-globular operad. $K$-algebras $KX \rightarrow X$ and $KY \rightarrow Y$ are equivalent if there exists a map of $K$-algebras $u:X \rightarrow Y$ or $u:Y \rightarrow X$ such that $u$ is surjective on $0$-cells, full on $m$-cells for all $1 \leq m \leq n$, and faithful on $n$.
\end{definition}
\noindent
But, as he said, this equivalence is not the best one: ``This definition of equivalence is much more (and thus much less general) than ought to be.'' To improve it, he suggested two approaches. The one of them is to replace the map $u$ with a span of maps of $K$-algebras. In this paper, we adopt this approach and prove two theorem. The first is that we define an adequate equivalence using spans
and prove this is indeed an equivalence relation. The second is that our equivalence is a generalization of ordinary equivalence between categories.

In Section 2 we recall the preliminary definitions (globular sets, thier maps, operads, algebras for a operad). In Section 3 we define the notion of \emph{span equivalence in $K \mathchar`- {\bf Alg}$} and prove the first theorem. In Section 4, for ordinary categories, we define \emph{span equivalence in ${\bf Cat}$} independently. Then we show that two categories are ordinary equivalent if and only if they are span equivalent in ${\bf Cat}$. To prove this, we use a combinatorial construction named \emph{equivalence fusion}. Futhermore, we show the second theorem.

\section{Preliminary}
The contents of the section is in [Cottrell 2015].

\begin{definition}
Let $n \in \mathbb{N}$. An {\em$n$-globular set} is a diagram 
\[
\xymatrix{
X = ( X_n \ar@<0.5ex>[r]^{s_n^X} \ar@<-0.5ex>[r]_{t_n^X} & X_{n-1} \ar@<0.5ex>[r]^{s_{n-1}^X} \ar@<-0.5ex>[r]_{t_{n-1}^X} 
 & ... \ar@<0.5ex>[r]^{s_1^X} \ar@<-0.5ex>[r]_{t_1^X} & X_{0}  )
 }
\]
of sets and maps such that
\[
s_{k-1}^X s_{k}^X (x) =s_{k-1}^X t_{k}^X (x) , \hspace{10pt} t_{k-1}^X s_{k}^X (x) = t_{k-1}^X t_{k}^X (x)
\]
for all $k \in \{ 2,...,n \} $ and $x \in X_k$. \\
Elements of $X_k$ are called {\em$k$-cells} of $X$. We defined {\em hom-sets} of $X$ as follows:
\[
{\bf Hom}_{X}(x,y) := \{ \alpha \in X_k \mid s_k^X(\alpha)=x, t_k^X(\alpha)=y \}
\] 
for all $k \in \{ 1,...n \} $ and $x,y \in X_{k-1}$.\\
Let $X,Y$ be $n$-globular sets, A \emph{map} of $n$-globular sets from $X$ to $Y$ is a collection $f=\{ f_{k}: X_{k} \rightarrow Y_{k} \}_{k \in \{ 1,...,n \}}$ of maps of sets such that 
\[
s_k^Y f_k (x) = f_{k-1} s_{k}^X (x) , \hspace{10pt} t_{k}^Y f_k (x) = f_{k-1} t_{k}^X (x)
\]
for all $k \in \{ 1,...,n \} $ and $x \in X_{k}$. \\
The category of $n$-globular sets and their maps is denoted by $n \mathchar`- {\bf GSet}$.
\end{definition}

\begin{definition}
A category is {\em cartesian} if it has all pullbacks. A functor is {\em cartesian} if it preserves pullbacks. A natural transformation is {\em cartesian} if it all of its naturality squares are pullbacks squares. A monad is {\em cartesian} if its functor part, unit and counit are cartesian. A map of monad is {\em cartesian} if its underlying natural transformation is cartesian.
\end{definition}

\begin{definition}
Let ${\cal C}$ be a cartesian category with a terminal object $1$. and $T$ be a cartesian monad on ${\cal C}$. The category of {\em $T$-collections} is the slice category ${\cal C}/T1$. The category has a monoidal structure: let $k: K \rightarrow T1, k': K' \rightarrow T1$ be collections; then their tensor product is defined to be the composite along the top of the diagram
\[ \xymatrix{
K \otimes K' \ar[r] \ar[d] \pullbackmark{0,1.5}{1.5,0} & TK' \ar[r]^{Tk'} \ar[d]^{T!} & T^2 1 \ar[r]^{\mu^T_1} & T1 \\
K \ar[r]_k & T1 & &
} \]
where $!$ is the unique map $K' \rightarrow 1$. the unit for this  tensor product is the collection
\[ \xymatrix{
1 \ar[d]^{\eta^T_1} \\
T1
} \]
The monoidal category is denoted by $T \mathchar`- {\bf Coll}$. 
\end{definition}

\begin{definition}
Let ${\cal C}$ be a cartesian category with a terminal object $1$, and $T$ be a cartesian monad on ${\cal C}$. A \emph{$T$-operad} is a monoid  in the monoidal category $T$-${\bf Coll}$. In the case in which $T$ is the free strict $n$-category monad on $n \mathchar`- {\bf GSet}$, a $T$-operad is called an \emph{$n$-globular operad}.
\end{definition}

\begin{definition}
Let ${\cal C}$ be a cartesian category with a terminal object $1$, $T$ be a cartesian monad on ${\cal C}$ and $K$ be a $T$-operad. Then there is an {\em induced monad} on ${\cal C}$, which by abuse of notation we denote $(K, \eta^K, \mu^K)$: The endfunctor
\[
K: {\cal C} \rightarrow {\cal C}
\]
is defined as follows; The object part of the functor, for $X \in {\cal C}$, $KX$ is defined by the pullback:
\[ \xymatrix{
KX \ar[r]^{K!} \ar[d]_{k_X} \pullbackmark{0,1.5}{1.5,0} & K \ar[d]^k \\
TX \ar[r]_{T!} & T1
} \]
The arrow part of the functor, for $Y \in {\cal C}, u: X \rightarrow Y$, $Ku$ is defined by the unique property of the pullback: \vspace{5pt}
\[ \xymatrix{
KX \ar@{.>}[r]^{Ku} \ar[d]_{k_X} \pullbackmark{0,1.5}{1.5,0} \ar@/^25pt/[rr]^{K!} & KY \ar[r]^{K!} \ar[d]_{K_Y} \pullbackmark{0,1.5}{1.5,0} & K \ar[d]^{k} \\
TX \ar[r]_{Tu} \ar@/_25pt/[rr]_{T!} & TY \ar[r]_{T!} & T1
} \] \vspace{10pt} \\
Components  $\eta^K_X, \mu^K_X$of the unit map $\eta^K: 1 \Rightarrow K$ and $\mu^K:K^2 \Rightarrow K$ are defined by the following diagrams: 
\[ \xymatrix{
X \ar@{.>}[rd]^{\eta^K_X} \ar[rr]^{!} \ar@/_10pt/[rdd]_{\eta^T_X} & & 1 \ar[d]^{\epsilon} \ar@/^20pt/[dd]^{\eta^T_1} \\
& KX \ar[r]^{K!} \ar[d]_{k_X} \pullbackmark{0,1.5}{1.5,0} & K \ar[d]^k \\
& TX \ar[r]_{T!} & T1
} \]
\[ \xymatrix@dr{
K^2 X \ar[r] \ar[d] \pullbackmark{1,1}{-1,1} \ar@{.>}@/^20pt/[rrdd]^{\mu^K_X} & TKX \ar[r] \ar[d] \pullbackmark{1,1}{-1,1} & T^2 X \ar[d]^{T^2 !} \ar@/^20pt/[rdd]^{\mu^T_X} & \\
K \otimes K \ar[r] \ar[d] \pullbackmark{1,1}{-1,1} \ar@/_70pt/[rrdd]_{\mu^K} & TK \ar[r]_{Tk} \ar[d]^{T!} & T^2 1 & \\
K \ar[r]_k & T1 & KX \ar[r] \ar[d] \pullbackmark{1,1}{-1,1} & TX \ar[d]^{T!} \\
& & K \ar[r]_{k} & T1 
} \]
\end{definition}

\begin{definition}
Let ${\cal C}$ be a cartesian category with a terminal object $1$, $T$ be a cartesian monad on ${\cal C}$ and $K$ be a $T$-operad. We define a \emph{$K$-algebra} as an algebra for the induced monad $(K, \eta^K , \mu^K )$. Similarly, a \emph{map} of algebras for $T$-operad $K$ is a map of algebras for the induced monad. The category of $K$-algebras and thier maps is denoted by $K \mathchar`- {\bf Alg}$.
\end{definition}
\noindent
Leinster's weak $n$-category is an algebra for specific operad. See section 9 and 10 in [Leinster 2004] for details.

\section{Span equivalence}

\begin{definition}
Let $f:X \rightarrow Y$ be a map of $n$-globular sets.
\begin{list}{$\bullet$}{}
\item $f$ is \emph{surjective on $k$-cells} $ : \Leftrightarrow$ $f_k : X_k \rightarrow Y_k $ is surjective
\item $f$ is \emph{injective on $k$-cells} $ : \Leftrightarrow$ $f_k : X_k \rightarrow Y_k $ is injective
\item $f$ is \emph{full on $k$-cells} $ : \Leftrightarrow$ 
$\left\{ \begin{array}{l} 
\forall x, x' \in X_{k-1}, \beta \in {\bf Hom}_{Y}(f_{k-1}(x), f_{k-1}(x')), \\
\exists \alpha \in {\bf Hom}_{X}(x,x') \hspace{5pt} {\rm s.t.} \hspace{5pt} f_{k}( \alpha )=\beta \\
\end{array} \right.$
\item $f$ is {\em faithful on $k$-cell} $ : \Leftrightarrow$ 
$\left\{ \begin{array}{l}{}
\forall x,x' \in X_{k-1}, \alpha ,\alpha ' \in {\bf Hom}_{X}(f_{k-1}(x),f_{k-1}(x')), \\
\alpha \neq \alpha ' \Rightarrow f_{k}(\alpha) \neq f_{k}(\alpha ' )
\end{array} \right.$
\end{list}
Let $f$ be a map of $K$-algebras. $f$ is \emph{surjective} (respectively, \emph{injective}, \emph{full}, \emph{faithful}) on $k$-cells if and only if the underlying map is surjective (respectively, injective, full, faithful) on $k$-cells.
\end{definition}

\begin{definition}
Let $K$ be an $n$-globular operad. $K$-algebras $\phi : KX \rightarrow X$ and $\psi : KY \rightarrow Y$ are {\em span equivalent in $K \mathchar`- {\bf Alg}$} if there exists a triple $\langle \theta , u, v \rangle$ such that $\theta : KZ \rightarrow Z$ is an $K$-algebra, $u: \theta \rightarrow \phi $ and $v: \theta \rightarrow \psi $ are maps of $K$-algebras, surjective on $0$-cells, full on $m$-cells for all $1 \leq m \leq n$, and faithful on $n$-cells. The triple $\langle \theta , u, v \rangle$ is referred to as an span equivalence of $K$-algebras.
\end{definition}
\noindent
Trivially, under the same situation as Theorem 1.1, the free $K$-algebra on $X$ is span equivalent to the free strict $n$-category on $X$. 

\begin{proposition}
In the pullback diagram in $n \mathchar`- {\bf GSet}$
\[ \xymatrix{
 P \pullbackmark{1,0}{0,1} \ar[r]^j \ar[d]_i & Y \ar[d]^g \\
 X \ar[r]_f & S
} \]
\begin{list}{$\bullet$}{}
\item $f$ is surjective on $0$-cells $\Rightarrow$ $j$ is surjective on $0$-cells
\item $f$ is full on $k$-cells $\Rightarrow$ $j$ is full on $k$-cells
\item $f$ is faithful on $k$-cells $\Rightarrow$ $j$ is faithful on $k$-cells
\end{list}
\end{proposition}

\begin{proof} We define an $n$-globular set $P$ as follows:
\[ P_k := \{ (x,y) \in X_k \times Y_k \mid f_k(x)=g_k(y) \}  \]
\[ s_l^P := ( P_l \ni (x,y) \mapsto (s_l^X (x), s_l^Y (y) ) \in P_{l-1} ) \]
\[ t_l^P := ( P_l \ni (x,y) \mapsto (t_l^X (x), t_l^Y (y) ) \in P_{l-1} ) \]
for all $k \in \{ 0,...,n \}, l \in \{ 1,...,n \}$, and maps of $n$-globular sets $i,j$ as follows:
\[
i_k := (P_k \ni (x,y) \mapsto x \in X_k), \hspace{10pt} j_k := (P_k \ni (x,y) \mapsto y \in Y_k)
\]
for all $k \in \{ 0,...,n \}$. Then $(P,i,j)$ is a pullback of $X$ and $Y$ over $S$. It is enough to prove the proposition that we check the claims for $(P,i,j)$.
Firstly, we prove surjectivity on $0$-cells. For $y \in Y_0$, there exists $x \in X_0$ such that $f_{0}(x)=g_{0}(y)$, So $(x,y) \in P_{0}$ and $j_{0}((x,y))=y$. which is the condition of surjectivity.
To show fullness, we suppose $(x,y), (x',y') \in P_{k-1}, \phi \in {\bf Hom}(y,y')$, we can see $s_{k} g_{k}(\phi)=g_{k-1}(y)=f_{k-1}, t_{k} g_{k} (\phi)=g_{k-1} (y')=f_{k-1} (x')$. Thus $g_{k} (\phi) \in {\bf Hom}(f_{k-1}(x), f_{k-1}(x'))$. For fullness, there exists $\psi \in {\bf Hom}(x,x')$ such that $f_k (\psi) = g_k (\phi)$. Then $(\psi , \phi ) \in {\bf Hom}((x,y),(x',y'))$ and $j_k (\psi ,\phi ) =\phi$. Therefore $j$ is full on $k$-cells. 
Lastly, we suppose that $f$ is faithful on $k$-cells. let $(x,y), (x',y') \in P_{k-1}$ and $ \psi, \phi \in {\bf Hom}((x,y), (x', y'))$ such that $j_{k}(\psi)=j_{k}(\phi)$. Then $f_k i_k(\psi)=g_k j_k (\psi)=g_k j_k (\phi)=f_k i_k(\phi)$. From faithfulness, $i_k (\psi)= i_k (\phi)$, and $\psi = (i_k (\psi ), j_k (\psi ))=(i_k (\phi ), j_k (\phi ))=\phi $. Therefore $j$ is faithful on $k$-cells.
\end{proof}

By the following remark, for the category of $K$-algebras, we can also get similar results of proposition 3.3.

\begin{remark}
Let $T$ be a monad on ${\cal C}$. Then the forgetful functor $U:K \mathchar`- {\bf Alg} \rightarrow {\cal C}$ creats limits. Hence any monadic functor reflects limits. (Theorem 3.4.2. in [TTT])
\end{remark}

\begin{proposition}
In $K \mathchar`- {\bf Alg}$. Let
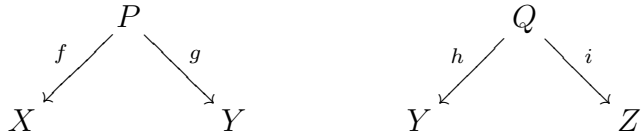
\[ \xymatrix{
 & \ar[ld]_f P \ar[rd]^g &  & & & \ar[ld]_h Q \ar[rd]^i & \\
X & & Y &  & Y & & Z 
} \]
be span equivalences, then
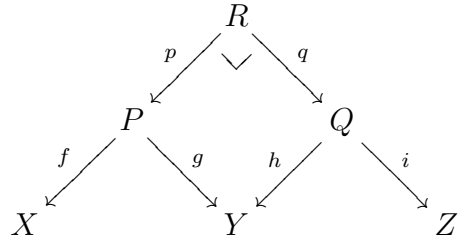
\[ \xymatrix{ 
 & & \ar[ld]_p R \pullbackmark{-1,1}{1,1} \ar[rd]^q & & \\
 & \ar[ld]_f P \ar[rd]^g &  & \ar[ld]_h Q \ar[rd]^i & \\
X & & Y & & Z 
} \]
is span equivalence.
\end{proposition}

\begin{proof}  By the fact, $p,q$ are are surjective on $0$-cells, full on $k$-cells for $1 \leq k \leq n$ and faithful on $n$-cells. Therefore $f \circ p , i \circ q$ are surjective on $0$-cells, full on $k$-cells for $1 \leq k \leq n$ and faithful on $n$-cells. So the span is span equivalence.
\end{proof}

\begin{theorem}
Span equivalence is equivalence relation on $K$-algebras.
\end{theorem}

\begin{proof} It is straightforward from the definition and previous proposition that span equivalence is equivalence relation.\end{proof}

\section{Characterizing equivalence of categories via spans}
In this section, we define span equivalence in ${\bf Cat}$ which is independent of that in $K \mathchar`- {\bf Alg}$. Then we show that two categories are ordinary equivalent if and only if they are span equivalent in {\bf Cat} and that span equivalence of categories implies span equivalence of algebras of them. Consequently, span equivalence is a generalization of ordinary equivalence.

\begin{definition}
Let ${\cal A}$ and ${\cal B}$ be categories. We say that ${\cal A}$ and ${\cal B}$ are \emph{span equivalent in {\bf Cat}} if there exists a triple $\langle {\cal A}, u, v \rangle$ such that ${\cal C}$ is a category, $u: {\cal C} \rightarrow {\cal A}$ and $v: {\cal C} \rightarrow {\cal B}$ are functors, surjective on objects, full and faithful.
\end{definition}

\begin{definition}
Let ${\cal A}$ and ${\cal B}$ be categories, let $\langle S:{\cal A}\rightarrow{\cal B}, T:{\cal B} \rightarrow {\cal A}, \eta :I_{{\cal A}} \rightarrow TS , \epsilon :ST \rightarrow I_{{\cal B}} \rangle$ be an adjoint equivalence between ${\cal A}$ and ${\cal B}$. We define a category, {\rm equivalence fusion} ${\cal A}\sqcup \hspace{-.76em} \mid {\cal B}$ , as follows:

\begin{list}{$\bullet$}{}
\item object-set
\[ {\bf Ob}{\rm (}{\cal A}\sqcup \hspace{-.76em} \mid {\cal B}{\rm )} := {\bf Ob}{\rm (}{\cal A}{\rm )} \bigsqcup {\bf Ob}{\rm (}{\cal B}{\rm )} \hspace{20pt}  {\rm (disjoint)} \]
\item hom-set
\[ {\bf Hom}{\rm (}x,y{\rm )} := \left\{ \begin{array}{ll}
\{ \langle f,x,y \rangle \mid f \in {\cal A}{\rm (}x,y{\rm )} \} & {\rm (}x,y \in {\cal A}{\rm )} \\
\{ \langle f,x,y \rangle \mid f \in {\cal B}{\rm (}x,y{\rm )} \} & {\rm (}x,y \in {\cal B}{\rm )} \\
\{ \langle f,x,y \rangle \mid f \in {\cal B}{\rm (}Sx,y{\rm )} \} & {\rm (}x \in {\cal A}, y \in {\cal B}{\rm )} \\
\{ \langle f,x,y \rangle \mid f \in {\cal B}{\rm (}x,Sy{\rm )} \} & {\rm (}x \in {\cal B}, y \in {\cal A}{\rm )} \\
\end{array} \right. \]
\item composition
\[ \begin{array}{ll}
\tilde{\circ} : {\bf Hom}{\rm (}y,z{\rm )} \times {\bf Hom}{\rm (}x,y{\rm )} & \longrightarrow {\bf Hom}{\rm (}x,z{\rm )} \\
\hspace{50pt} \langle \langle g,y,z \rangle , \langle f,x,y \rangle \rangle & \longmapsto \langle g,y,z \rangle \tilde{\circ} \langle f,x,y \rangle := \langle g \circ f,x,z \rangle 
\end{array} \]
 \[ g \circ f := \left\{ \begin{array}{ll}
 g \circ_{{\cal A}} f & {\rm (}x,y,z \in {\cal A}{\rm )} \\
 g \circ_{{\cal B}} f & {\rm (}x,y,z \in {\cal B}{\rm )} \\
 g \circ_{{\cal B}} Sf & {\rm (}x,y \in {\cal A}, z \in {\cal B}{\rm )} \\
 g \circ_{{\cal B}} f & {\rm (}x \in {\cal A}, y,z \in {\cal B}{\rm )} \\
 g \circ_{{\cal B}} f & {\rm (}x,y \in {\cal B}, z \in {\cal A}{\rm )} \\
 Sg \circ_{{\cal B}} f & {\rm (}x \in {\cal B}, y,z \in {\cal A}{\rm )} \\
 \eta_{z}^{-1} \circ_{{\cal A}} Tg \circ_{{\cal A}} Tf \circ_{{\cal A}} \eta_{x} & {\rm (}x \in {\cal A}, y \in {\cal B}, z \in {\cal A}{\rm )} \\
 g \circ_{{\cal B}} f & {\rm (}x \in {\cal B}, y \in {\cal A}, z \in {\cal B}{\rm )} \\
\end{array} \right. \]
\item identities
\[ {\rm id}_{x} :=\left\{ \begin{array}{ll}
\langle {\rm id}_{x},x,x \rangle & {\rm (}x \in {\cal A}, {\rm id}_{x} \in {\cal A}{\rm (}x,x{\rm )} {\rm )} \\
\langle {\rm id}_{x},x,x \rangle & {\rm (}x \in {\cal B}, {\rm id}_{x} \in {\cal B}{\rm (}x,x{\rm )} {\rm )} \\
\end{array} \right. \]
\end{list}
\end{definition}

\begin{proposition}
The equivalence fusion ${\cal A}\sqcup \hspace{-.76em} \mid {\cal B}$ forms a category.
\end{proposition}

\begin{proof}
It is easy to check that the composition $\tilde{\circ}$ is map from ${\bf Hom}(x,y) \times {\bf Hom}(y,z)$ to ${\bf Hom}(x,z)$. Now, we prove that the composition $\tilde{\circ}$ satisfies associative law and identity law by case analysis.
\begin{itemize}

\item associative law
 \begin{itemize}
 \item $x \in {\cal A}, y \in {\cal A}, z \in {\cal A}, w \in {\cal A}$, \\
 $h \circ ( g \circ f ) = h \circ_{{\cal A}} ( g \circ_{{\cal A}} f ) $\\
 $( h \circ g ) \circ f = ( h \circ_{{\cal A}} g ) \circ_{{\cal A}} f $
 \item $x \in {\cal A}, y \in {\cal A}, z \in {\cal A}, w \in {\cal B}$, \\
 $h \circ ( g \circ f ) = h \circ ( g \circ_{{\cal A}} f ) = h \circ_{{\cal B}} S(g \circ_{{\cal A}} f) = h \circ_{{\cal B}} ( Sg \circ_{{\cal B}} Sf )$\\
 $( h \circ g ) \circ f = ( h \circ_{{\cal B}} Sg ) \circ f = ( h \circ_{{\cal B}} Sg ) \circ_{{\cal B}} Sf $
 \item $x \in {\cal A}, y \in {\cal A}, z \in {\cal B}, w \in {\cal A}$, \\
 $h \circ ( g \circ f ) = h \circ ( g \circ_{{\cal B}} Sf ) = \eta_{w}^{-1} \circ_{{\cal A}} Th \circ_{{\cal A}} T(g \circ_{{\cal B}} Sf) \circ_{{\cal A}} \eta_{x}  \\
 \hspace{134pt} =\eta_{w}^{-1} \circ_{{\cal A}} Th \circ_{{\cal A}} Tg \circ_{{\cal A}} TSf \circ_{{\cal A}} \eta_{x} \\
  \hspace{134pt} =\eta_{w}^{-1} \circ_{{\cal A}} Th \circ_{{\cal A}} Tg \circ_{{\cal A}} \eta_{y} \circ_{\cal A} f  \\
 ( h \circ g ) \circ f = ( \eta_{w}^{-1} \circ_{{\cal A}} Th \circ_{{\cal A}} Tg \circ_{{\cal A}} \eta_{y} ) \circ f =( \eta_{w}^{-1} \circ_{{\cal A}} Th \circ_{{\cal A}} Tg \circ_{{\cal A}} \eta_{y} ) \circ_{{\cal A}} f $
 \item $ x \in {\cal A}, y \in {\cal A}, z \in {\cal B}, w \in {\cal B}$, \\
 $h \circ ( g \circ f ) = h \circ ( g \circ_{{\cal B}} Sf ) = h \circ_{{\cal B}} ( g \circ_{{\cal B}} Sf )$\\
$( h \circ g ) \circ f = ( h \circ_{{\cal B}} g ) \circ f =  ( h \circ_{{\cal B}} g ) \circ_{{\cal B}} Sf $
 \item $x \in {\cal A}, y \in {\cal B}, z \in {\cal A}, w \in {\cal A}$, \\
 $h \circ ( g \circ f ) = h \circ ( \eta_{z}^{-1} \circ_{{\cal A}} Tg  \circ_{{\cal A}}  Tf  \circ_{{\cal A}}  \eta_{x}) = h  \circ_{{\cal A}}   \eta_{z}^{-1}  \circ_{{\cal A}} Tg  \circ_{{\cal A}} Tf  \circ_{{\cal A}} \eta_{x} \\
 \hspace{210pt} = \eta_{w}^{-1} \circ_{{\cal A}} TSh \circ_{{\cal A}} Tg  \circ_{{\cal A}} Tf  \circ_{{\cal A}} \eta_{x} $ \\
$( h \circ g ) \circ f = ( Sh \circ_{{\cal B}} g ) \circ f = \eta_{w}^{-1} \circ_{{\cal A}} T(Sh \circ_{{\cal B}} g) \circ_{{\cal A}} Tf \circ_{{\cal A}} \eta_{x} \\
\hspace{133pt} = \eta_{w}^{-1} \circ_{{\cal A}} TSh \circ_{{\cal A}} Tg \circ_{{\cal A}} Tf \circ_{{\cal A}} \eta_{x}$
 \item $x \in {\cal A}, y \in {\cal B}, z \in {\cal A}, w \in {\cal B}$, \\
 $h \circ ( g \circ f ) = h \circ ( \eta_{z}^{-1} \circ_{{\cal A}} Tg \circ_{{\cal A}} Tf \circ_{{\cal A}} \eta_{x} )\\
\hspace{52pt} = h \circ_{{\cal B}} S( \eta_{z}^{-1} \circ_{{\cal A}} Tg \circ_{{\cal A}} Tf \circ_{{\cal A}} \eta_{x} )\\
\hspace{52pt} =  h \circ_{{\cal B}} S\eta_{z}^{-1} \circ_{{\cal B}} ST(g \circ_{{\cal B}} f) \circ_{{\cal B}} S\eta_{x} \\
\hspace{52pt} = h \circ_{\cal B} (\epsilon_{Sz} \circ_{{\cal B}} S \eta_{z}) \circ_{{\cal B}}S \eta_{z}^{-1} \circ_{{\cal B}} ST(g \circ_{{\cal B}} f) \circ_{{\cal B}} S \eta_{x} \\
\hspace{52pt} = h \circ_{\cal B} \epsilon_{Sz} \circ_{{\cal B}} ST(g \circ_{{\cal B}} f) \circ_{{\cal B}} S \eta_{x} \\
\hspace{52pt} = h \circ_{\cal B} g \circ_{{\cal B}} f \circ_{{\cal B}} \epsilon_{Sx} \circ_{{\cal B}} S \eta_{x} \\
\hspace{52pt} = h \circ_{\cal B} g \circ_{{\cal B}} f $\\
$( h \circ g ) \circ f = ( h \circ_{{\cal B}} g ) \circ f = ( h \circ_{{\cal B}} g ) \circ_{{\cal B}} f $
 \item $ x \in {\cal A}, y \in {\cal B}, z \in {\cal B}, w \in {\cal A}$, \\
 $h \circ ( g \circ f ) = h \circ ( g \circ_{{\cal B}} f ) = \eta_{w}^{-1} \circ_{{\cal A}} Th \circ_{{\cal A}} T(g \circ_{{\cal B}} f) \circ_{{\cal A}} \eta_{x}$\\
$( h \circ g ) \circ f = ( h \circ_{{\cal B}} g ) \circ f = \eta_{w}^{-1} \circ_{{\cal A}} T(h \circ_{{\cal B}} g) \circ_{{\cal A}} Tf \circ_{{\cal A}} \eta_{x} $
 \item $x \in {\cal A}, y \in {\cal B}, z \in {\cal B}, w \in {\cal B}$, \\
$h \circ ( g \circ f ) = h \circ_{{\cal B}} ( g \circ_{{\cal B}} f ) $\\
$( h \circ g ) \circ f = ( h \circ_{{\cal B}} g ) \circ_{{\cal B}} f $
 \item $x \in {\cal B}, y \in {\cal A}, z \in {\cal A}, w \in {\cal A}$, \\
$h \circ ( g \circ f ) = h \circ (Sg \circ_{{\cal B}} f) = Sh \circ_{{\cal B}} (Sg \circ_{{\cal B}} f)$\\
$( h \circ g ) \circ f = (h \circ_{{\cal A}} g) \circ f = S(h \circ_{{\cal A}} g) \circ_{{\cal B}} f = (Sh \circ_{{\cal B}} Sg) \circ_{{\cal B}} f$
 \item $x \in {\cal B}, y \in {\cal A}, z \in {\cal A}, w \in {\cal B}$, \\
$h \circ ( g \circ f ) = h \circ (Sg \circ_{{\cal B}} f) = h \circ_{{\cal B}} (Sg \circ_{{\cal B}} f)$\\
$( h \circ g ) \circ f = (h \circ_{{\cal B}} Sg) \circ f = (h \circ_{{\cal B}} Sg) \circ_{{\cal B}} f$
 \item $x \in {\cal B}, y \in {\cal A}, z \in {\cal B}, w \in {\cal A}$, \\
$h \circ ( g \circ f ) = h \circ (g \circ_{{\cal B}} f) = h \circ_{{\cal B}} (g \circ_{{\cal B}} f)$\\
$( h \circ g ) \circ f = (\eta_{w}^{-1} \circ_{{\cal A}} Th \circ_{{\cal A}} Tg \circ_{{\cal A}} \eta_{y}) \circ f \\
\hspace{52pt} = S (\eta_{w}^{-1} \circ_{{\cal A}} Th \circ_{{\cal A}} Tg \circ_{{\cal A}} \eta_{y}) \circ_{{\cal B}} f \\
\hspace{52pt} = S\eta_{w}^{-1} \circ_{{\cal B}} ST(h \circ_{{\cal B}} g) \circ_{{\cal B}} S\eta_{y} \circ_{{\cal B}} f \\
\hspace{52pt} = (\epsilon_{Sw} \circ_{{\cal B}} S \eta_{w}) \circ_{{\cal B}}S \eta_{w}^{-1} \circ_{{\cal B}} ST(h \circ_{{\cal B}} g) \circ_{{\cal B}} S \eta_{y} \circ_{\cal B} f \\
\hspace{52pt} = \epsilon_{Sw} \circ_{{\cal B}} ST(h \circ_{{\cal B}} g) \circ_{{\cal B}} S \eta_{y} \circ_{\cal B} f \\
\hspace{52pt} = h \circ_{{\cal B}} g \circ_{{\cal B}} \epsilon_{Sy} \circ_{{\cal B}} S \eta_{y} \circ_{\cal B} f \\
\hspace{52pt} = h \circ_{{\cal B}} g \circ_{\cal B} f$
 \item $x \in {\cal B}, y \in {\cal A}, z \in {\cal B}, w \in {\cal B}$, \\
$h \circ ( g \circ f ) = h \circ_{{\cal B}} ( g \circ_{{\cal B}} f ) $\\
$( h \circ g ) \circ f = ( h \circ_{{\cal B}} g ) \circ_{{\cal B}} f $
 \item $x \in {\cal B}, y \in {\cal B}, z \in {\cal A}, w \in {\cal A}$, \\
$h \circ ( g \circ f ) = h \circ (g \circ_{{\cal B}} f) = Sh \circ_{{\cal B}} (g \circ_{{\cal B}} f)$\\
$( h \circ g ) \circ f = (Sh \circ_{{\cal B}} g) \circ f = (Sh \circ_{{\cal B}} g) \circ_{{\cal B}} f$
 \item $x \in {\cal B}, y \in {\cal B}, z \in {\cal A}, w \in {\cal B}$, \\
$h \circ ( g \circ f ) = h \circ_{{\cal B}} ( g \circ_{{\cal B}} f ) $\\
$( h \circ g ) \circ f = ( h \circ_{{\cal B}} g ) \circ_{{\cal B}} f $
 \item $x \in {\cal B}, y \in {\cal B}, z \in {\cal B}, w \in {\cal A}$, \\
$h \circ ( g \circ f ) = h \circ_{{\cal B}} ( g \circ_{{\cal B}} f ) $\\
$( h \circ g ) \circ f = ( h \circ_{{\cal B}} g ) \circ_{{\cal B}} f $
 \item $x \in {\cal B}, y \in {\cal B}, z \in {\cal B}, w \in {\cal B}$, \\
$h \circ ( g \circ f ) = h \circ_{{\cal B}} ( g \circ_{{\cal B}} f ) $\\
$( h \circ g ) \circ f = ( h \circ_{{\cal B}} g ) \circ_{{\cal B}} f $
 \end{itemize}
 
\item identity law
 \begin{itemize}
 \item $ x \in {\cal A}, y \in {\cal A}$, \\
$f \circ {\rm id}_{x} = f \circ_{{\cal A}} {\rm id}_{x} = f$\\
${\rm id}_{y} \circ f = {\rm id}_{y} \circ_{{\cal A}} f = f$
 \item $x \in {\cal A}, y \in {\cal B}$, \\
$f \circ {\rm id}_{x} = f \circ_{{\cal B}} S{\rm id}_{x} = f \circ_{{\cal B}} {\rm id}_{Sx} = f$\\
${\rm id}_{y} \circ f = {\rm id}_{y} \circ_{{\cal B}} f = f$
 \item $ x \in {\cal B}, y \in {\cal A}$, \\
$f \circ {\rm id}_{x} = f \circ_{{\cal B}} {\rm id}_{x} = f$\\
${\rm id}_{y} \circ f = S{\rm id}_{y} \circ_{{\cal B}} f = {\rm id}_{Sy} \circ_{{\cal B}} f = f$
 \item $x \in {\cal B}, y \in {\cal B}$, \\
$f \circ {\rm id}_{x} = f \circ_{{\cal B}} {\rm id}_{x} = f$\\
${\rm id}_{y} \circ f = {\rm id}_{y} \circ_{{\cal B}} f = f$
 \end{itemize}

\end{itemize}
\end{proof}

\begin{definition}
Let $\langle S:{\cal A}\rightarrow{\cal B}, T:{\cal B} \rightarrow {\cal A},  \eta :I_{{\cal A}} \rightarrow TS , \epsilon :ST \rightarrow I_{{\cal B}} \rangle$ be an adjoint equivalence, let ${\cal A}\sqcup \hspace{-.76em} \mid {\cal B}$ be the equivalence fusion. We define the {\em projections} $u,v$ as follows:

\begin{list}{$\bullet$}{}
\item $u:{\cal A}\sqcup \hspace{-.76em} \mid {\cal B} \longrightarrow A$ \\
 object-function 
$ u:{\bf Ob}{\rm (}{\cal A}\sqcup \hspace{-.76em} \mid {\cal B}{\rm )} \longrightarrow {\bf Ob}{\rm (}{\cal A}{\rm )} $ \\
$ \hspace{141pt} x  \longmapsto ux := \left\{ \begin{array}{ll}
x & {\rm (}x \in {\cal A}{\rm )} \\
Tx & {\rm (}x \in {\cal B}{\rm )} \\
\end{array} \right. $ \\
 hom-functions $u:{\bf Hom}{\rm (}x,y{\rm )} \longrightarrow {\cal A}{\rm (}ux,uy{\rm )}$ \\
\hspace{106pt} $\langle f,x,y \rangle \longmapsto uf := \left\{ \begin{array}{ll} 
f & {\rm (}x,y \in {\cal A}{\rm )} \\
Tf & {\rm (}x,y \in {\cal B}{\rm )} \\
Tf \circ_{{\cal A}} \eta_{x} & {\rm (}x \in {\cal A}, y \in {\cal B}{\rm )} \\
\eta_{y}^{-1} \circ_{{\cal A}} Tf & {\rm (}x \in {\cal B}, y \in {\cal A}{\rm )} \\
\end{array} \right. $ \\

\item $v:{\cal A}\sqcup \hspace{-.76em} \mid {\cal B} \longrightarrow B$ \\
 object-function $v:{\bf Ob}{\rm (}{\cal A}\sqcup \hspace{-.76em} \mid {\cal B}{\rm )} \longrightarrow {\bf Ob}{\rm (}{\cal B}{\rm )}$ \\
\hspace{141pt} $x \longmapsto vx := \left\{ \begin{array}{ll}
Sx & {\rm (}x \in {\cal A}{\rm )} \\
x & {\rm (}x \in {\cal B}{\rm )} \\
\end{array} \right. $ \\
 hom-functions $v:{\bf Hom}{\rm (}x,y{\rm )} \longrightarrow {\cal B}{\rm (}ux,uy{\rm )}$　\\
\hspace{106pt} $\langle f,x,y \rangle \longmapsto vf := \left\{ \begin{array}{ll}
Sf & {\rm (}x,y \in {\cal A}{\rm )} \\
f & {\rm ( others )} \\ 
\end{array} \right.$ 
\end{list}
\end{definition}

\begin{proposition}
The projections $u,v$ are functors.
\end{proposition}

\begin{proof} We show that $u,v$ preserve composition of morphisms and identity morphism by case analysis.
\begin{itemize}
\item $u$ preserves composition of morphisms
 \begin{itemize}
 \item $x \in {\cal A}, y \in {\cal A}, z \in {\cal A}$, \\
 $u(g \circ f) = u(g \circ_{{\cal A}} f) = g \circ_{{\cal A}} f$ \\
$ug \circ_{{\cal A}} uf = g \circ_{{\cal A}} f$ 
 \item $x \in {\cal A}, y \in {\cal A}, z \in {\cal B}$, \\
 $u(g \circ f) = u(g \circ_{{\cal B}} Sf) = T(g \circ_{{\cal B}} Sf) \circ_{{\cal A}} \eta_{x}=Tg \circ_{{\cal A}} TSf \circ_{{\cal A}} \eta_{x}$ \\
$ug \circ_{{\cal A}} uf =(Tg \circ_{{\cal A}} \eta_{y}) \circ_{{\cal A}} f = Tg \circ_{\cal A} TSf \circ_{{\cal A}} \eta_{x}$ 
 \item $x \in {\cal A}, y \in {\cal B}, z \in {\cal A}$, \\
 $u(g \circ f) =u(\eta_{z}^{-1} \circ_{{\cal A}} Tg \circ_{{\cal A}} Tf \circ_{{\cal A}} \eta_{x})=\eta_{z}^{-1} \circ_{{\cal A}} Tg \circ_{{\cal A}} Tf \circ_{{\cal A}} \eta_{x}$ \\
$ug \circ_{{\cal A}} uf=(\eta_{z}^{-1} \circ_{{\cal A}} Tg) \circ_{{\cal A}} (Tf \circ_{{\cal A}} \eta_{x})$ 
 \item $x \in {\cal A}, y \in {\cal B}, z \in {\cal B}$, \\
 $u(g \circ f) = u(g \circ_{{\cal B}} f)=T(g \circ_{{\cal B}} f) \circ_{{\cal A}} \eta_{x}=Tg \circ_{{\cal A}} Tf \circ_{{\cal A}} \eta_{x}$ \\
$ug \circ_{{\cal A}} uf = Tg \circ_{{\cal A}} (Tf \circ_{{\cal A}} \eta_{x})$
 \item $x \in {\cal B}, y \in {\cal A}, z \in {\cal A}$, \\
$u(g \circ f) = u(Sg \circ_{{\cal B}} f) = \eta_{z}^{-1} \circ_{{\cal A}} T(Sg \circ_{{\cal B}} f) = \eta_{z}^{-1} \circ_{{\cal A}} TSg \circ_{{\cal B}} Tf$ \\
$ug \circ_{{\cal A}} uf = g \circ_{{\cal A}} (\eta_{y}^{-1} \circ_{{\cal A}} Tf) = \eta_{z}^{-1} \circ_{{\cal A}} TSg \circ_{\cal A} Tf$
 \item $ x \in {\cal B}, y \in {\cal A}, z \in {\cal B}$, \\
 $u(g \circ f) = u(g \circ_{{\cal B}} f) = T(g \circ_{{\cal B}} f) = Tg \circ_{{\cal A}} Tf$ \\
$ug \circ_{{\cal A}} uf = (Tg \circ_{{\cal A}} \eta_{y}) \circ_{{\cal A}} (\eta_{y}^{-1} \circ_{{\cal A}} Tf)= Tg \circ_{{\cal A}} Tf$
 \item $ x \in {\cal B}, y \in {\cal B}, z \in {\cal A}$, \\
$u(g \circ f) = u(g \circ_{{\cal B}} f) = \eta_{z}^{-1} \circ_{{\cal A}} T(g \circ_{{\cal B}} f) = \eta_{z}^{-1} \circ_{{\cal A}} Tg \circ_{{\cal A}} Tf$ \\
$ug \circ_{{\cal A}} uf = (\eta_{z}^{-1} \circ_{{\cal A}} Tg) \circ_{{\cal A}} Tf$
 \item $x \in {\cal B}, y \in {\cal B}, z \in {\cal B}$, \\
 $u(g \circ f) = u( g \circ_{{\cal B}} f ) = T(g \circ_{{\cal B}} f)=Tg \circ_{{\cal A}} Tf$ \\
$ug \circ_{{\cal A}} uf =Tg \circ_{{\cal A}} Tf$
 \end{itemize}
 
\item $u$ preserves identity morphisms
 \begin{itemize}
 \item $x \in {\cal A}$, \\
$u({\rm id}_{x})={\rm id}_{x}={\rm id}_{ux}$
 \item $x \in {\cal B}$, \\
$u({\rm id}_{x})=T{\rm id}_{x}={\rm id}_{Tx}={\rm id}_{ux}$
 \end{itemize}
 
 \item $v$ preserves composition of morphisms
 \begin{itemize}
 \item $x \in {\cal A}, y \in {\cal A}, z \in {\cal A}$, \\
$v(g \circ f) =v(g \circ_{{\cal A}} f)=S(g \circ_{{\cal A}} f) = Sg \circ_{{\cal B}} Sf$ \\
$vg \circ_{{\cal B}} vf =Sg \circ_{{\cal A}} Sf$
 \item $x \in {\cal A}, y \in {\cal A}, z \in {\cal B}$, \\
$v(g \circ f) = v(g \circ_{{\cal B}} Sf) = g \circ_{{\cal B}} Sf$ \\
$vg \circ_{{\cal B}} vf = g \circ_{{\cal B}} Sf$
 \item $ x \in {\cal A}, y \in {\cal B}, z \in {\cal A}$, \\
$v(g \circ f) = v(\eta_{z}^{-1} \circ_{{\cal A}} Tg \circ_{{\cal A}} Tf \circ_{{\cal A}} \eta_{x}) \\
\hspace{40pt} =S(\eta_{z}^{-1} \circ_{{\cal A}} Tg \circ_{{\cal A}} Tf \circ_{{\cal A}} \eta_{x}) \\
\hspace{40pt} = S\eta_{z}^{-1} \circ_{{\cal B}} ST(g \circ_{{\cal B}} f) \circ_{{\cal B}} S\eta_{x} \\
\hspace{40pt} = g \circ_{{\cal B}} f $ \\
$vg \circ_{{\cal B}} vf = g \circ_{{\cal B}} f$
\item $ x \in {\cal A}, y \in {\cal B}, z \in {\cal B}$, \\
$v(g \circ f) = v(g \circ_{{\cal B}} f) = g \circ_{{\cal B}} f$ \\
$vg \circ_{{\cal B}} vf = g \circ_{{\cal B}} f$
\item $x \in {\cal B}, y \in {\cal A}, z \in {\cal A}$, \\
$v(g \circ f) = v(Sg \circ_{{\cal B}} f) = Sg \circ_{{\cal B}} f$ \\
$vg \circ_{{\cal B}} vf = Sg \circ_{{\cal B}} f$
\item $x \in {\cal B}, y \in {\cal A}, z \in {\cal B}$, \\
$v(g \circ f) = v(g \circ_{{\cal B}} f) = g \circ_{{\cal B}} f$ \\
$vg \circ_{{\cal B}} vf = g \circ_{{\cal B}} f$
\item $x \in {\cal B}, y \in {\cal B}, z \in {\cal A}$, \\
$v(g \circ f) = v(g \circ_{{\cal B}} f) = g \circ_{{\cal B}} f$ \\
$vg \circ_{{\cal B}} vf = g \circ_{{\cal B}} f$ 
\item $ x \in {\cal B}, y \in {\cal B}, z \in {\cal B}$, \\
$v(g \circ f) = v(g \circ_{{\cal B}} f) = g \circ_{{\cal B}} f$ \\
$vg \circ_{{\cal B}} vf = g \circ_{{\cal B}} f$
 \end{itemize}
 
\item $v$ preserves identity morphisms
 \begin{itemize}
 \item $ x \in {\cal A}$ \\
$v({\rm id}_{x})=S{\rm id}_{x}={\rm id}_{Sx}={\rm id}_{vx}$
 \item $x \in {\cal B}$ \\
$v({\rm id}_{x})={\rm id}_{x}~{\rm id}_{vx}$
 \end{itemize}
 
\end{itemize}
\end{proof}

\begin{proposition}
The projections $u,v$ are surjective on objects, full and faithful.
\end{proposition}

\begin{proof} It's trivial by definitions that $u,v$ are surjective on objects. So we check fullness and faithfulness.
\begin{itemize}
 \item $u$ is full and faithful
 \begin{itemize}
 \item $x,y \in {\cal A}$, \\
 $u:{\bf Hom}(x,y)=\{ \langle f,x,y \rangle \mid f \in {\cal A}(x,y) \} \ni \langle f,x,y \rangle \mapsto f \in {\cal A}(x,y)$ is bijective.
 \item $x,y \in {\cal B}$, \\
 $T:{\cal B}{\rm (}x,y{\rm )} \rightarrow {\cal A}{\rm (}Tx,Ty{\rm )}$ is bijective. Therefore $u: {\bf Hom}{\rm (}x,y{\rm )}=\{ \langle f,x,y \rangle \mid f \in {\cal B}(x,y) \} \ni \langle f,x,y \rangle \mapsto f \in {\cal A}(x,y) \ni \langle f,x,y \rangle \mapsto Tf \in {\cal A}{\rm (}Tx,Ty{\rm )}={\cal A}{\rm (}ux,uy{\rm )}$ is bijective.
 \item $x \in {\cal A}, y \in {\cal B}$, \\
${\cal B}{\rm (}Sx,y{\rm )} \ni f \mapsto Tf \circ_{{\cal A}} \eta_{x} \in {\cal A}{\rm (}x,Ty{\rm )} $ is the right adjunct of each $f$, and bijective. Therefore $u: {\bf Hom}{\rm (}x,y{\rm )}=\{ \langle f,x,y \rangle \mid f \in {\cal B}(Sx,y) \} \ni \langle f,x,y \rangle \mapsto Tf \circ_{{\cal A}} \eta_{x} \in {\cal A}{\rm (}x,Ty{\rm )}={\cal A}{\rm (}ux,uy{\rm )}$ is bijective.
 \item $x \in {\cal B} , y \in {\cal A}$, \\
${\cal B}{\rm (}x,Sy{\rm )} \ni f \mapsto \eta_{y}^{-1} \circ_{{\cal A}} Tf \in {\cal A}{\rm (}Tx,y{\rm )}$ is the left adjunct of each $f$, and bijective. Therefore $u: {\bf Hom}{\rm (}x,y{\rm )}=\{ \langle f,x,y \rangle \mid f \in {\cal B}(x,Sy) \} \ni \langle f,x,y \rangle \mapsto \eta_{y}^{-1} \circ_{{\cal A}} Tf \in {\cal A}{\rm (}Tx,y{\rm )}={\cal A}{\rm (}ux,uy{\rm )}$
 \end{itemize}
 
 \item $v$ is full and faithful
 \begin{itemize}
 \item $x,y \in {\cal A}$, \\
$S:{\cal A}{\rm (}x,y{\rm )} \rightarrow {\cal B}{\rm (}Sx,Sy{\rm )}$ is bijective. Therefore $v:{\bf Hom}{\rm (}x,y{\rm )}=\{ \langle f,x,y \rangle \mid f \in {\cal A}(x,y) \} \ni \langle f,x,y \rangle \mapsto Sf \in {\cal B}{\rm (}Sx,Sy{\rm )}={\cal B}{\rm (}vx,vy{\rm )}$ is bijective.
 \item $x,y \in {\cal B}$, \\
$v:{\bf Hom}{\rm (}x,y{\rm )}=\{ \langle f,x,y \rangle \mid f \in {\cal B}(x,y) \} \ni \langle f,x,y \rangle \mapsto f \in {\cal B}{\rm (}x,y{\rm )}={\cal B}{\rm (}vx,vy{\rm )}$ is bijective.
 \item $x \in {\cal A}, y \in {\cal B}$, \\
$v:{\bf Hom}{\rm (}x,y{\rm )}=\{ \langle f,x,y \rangle \mid f \in {\cal B}(Sx,y) \} \ni \langle f,x,y \rangle \mapsto f \in {\cal B}{\rm (}Sx,y{\rm )}={\cal B}{\rm (}vx,vy{\rm )}$ is bijective.
 \item $x \in {\cal B} , y \in {\cal A}$, \\
$v:{\bf Hom}{\rm (}x,y{\rm )}=\{ \langle f,x,y \rangle \mid f \in {\cal B}(x,Sy) \} \ni \langle f,x,y \rangle \mapsto f \in {\cal B}{\rm (}x,Sy{\rm )}={\cal B}{\rm (}vx,vy{\rm )}$ is bijective.
 \end{itemize}
 
\end{itemize}
\end{proof}

\begin{theorem}
Let ${\cal A}$ and ${\cal B}$ be categories. ${\cal A}$ is equivalent to ${\cal B}$ 
if and only if ${\cal A}$ is span equivalent to ${\cal B}$ in ${\bf Cat}$. 
\end{theorem}

\begin{proof} Let ${\cal A}$ be equivalent to ${\cal B}$, then ${\cal A}$ is adjoint equivalent to ${\cal B}$. Thus there exists a adjoint equivalence between ${\cal A}$ and ${\cal B}$. So we can construct the equivalence fusion and the projections. By Propositions, they are span equivalence in ${\bf Cat}$. Therefore ${\cal A}$ is span equivalent to ${\cal B}$. \\
On the other hand, let ${\cal A}$ be span equivalent to ${\cal B}$ in ${\bf Cat}$. Then there exists a span equivalence $\langle {\cal C},u,v \rangle$ between ${\cal A}$ and ${\cal B}$, and ${\cal C}$ is equivalent to both ${\cal A}$ and ${\cal B}$. Therefore ${\cal A}$ is equivalent to ${\cal B}$.
\end{proof}

\begin{remark}
Let ${\cal A}$ be presheaf category. The forgetful functor
\[
U: {\cal A} \mathchar`- {\bf Cat} \longrightarrow {\cal A} \mathchar`- {\bf Gph}
\]
is monadic. (Proposition F 1.1 in [Leinster 2004])
\end{remark}
\noindent
Let ${\cal A} = {\bf Set}$, we can see ${\bf Set} \mathchar`- {\bf Cat} = {\bf Cat}$, ${\bf Set} \mathchar`- {\bf Grp} = 1 \mathchar`- {\bf GSet}$, and the induced monad $T_1$ is the free strict $1$-category monad on $1 \mathchar`- {\bf GSet}$. By the remark, the comparison functor 
\[
N: {\bf Cat} \longrightarrow T_1 \mathchar`- {\bf Alg}
\]
is isomorphic and arrow part of the functor is
\[
N: f \longmapsto Uf .
\]
Moreover, the category ${\bf Wk} \mathchar`- 1 \mathchar`- {\bf Cat}$ of Leinster's weak 1 categories is the category $T_1 \mathchar`- {\bf Alg}$ of algebras for the monad for details,  refer to the proof of Theorem 9.1.4 in [Leinster 2004].
So the isomorphism $N: {\bf Cat} \rightarrow {\bf Wk} \mathchar`- 1 \mathchar`- {\bf Cat}$ preserve surjectivity, fullness and faithfullness. Hence,

\begin{proposition}
Let $N: {\bf Cat} \rightarrow {\bf Wk \mathchar`- 1 \mathchar`- Cat}$ be the isomorphism above. let ${\cal A}$ and ${\cal B}$ be categories. ${\cal A}$ is span equivalent to ${\cal B}$ in ${\bf Cat}$ if and only if $N({\cal A})$ is span equivalent to $N({\cal B})$ in ${\bf Wk \mathchar`- 1 \mathchar`- Cat}$.
\end{proposition}

\noindent
As a result of Proposition 4.7 and Proposition 4.9, we obtain the following theorem:

\begin{theorem}
${\cal A}$ is equivalent to ${\cal B}$ if and only if $N({\cal A})$ is span equivalent to $N({\cal B})$ in ${\bf Wk \mathchar`- 1 \mathchar`- Cat}$.
\end{theorem}



\vspace{5mm}
\noindent
Division of Science (mathematics), \\
Graduate School, Kyoto Sango University, \\
Kyoto 603-8555, Japan \\
E-mail: i1655059@cc.kyoto-su.ac.jp



\begin{thebibliography}{99}
 \bibitem{C1}[Cottrell 2015] Thomas Cottrell, OPERADIC DEFINITIONS OF WEAK N-CATEGORY: COHERENCE AND COMPARISONS, Theory and Applications of Categories, Vol. 30, No. 13, 2015, pp. 433-488.
 \bibitem{L1}[Leinster 2004] Tom Leinster, Higher operads, higher categories, volume 298 of London Math-
ematical Society Lecture Note Series. Cambridge University Press, Cambridge, 2004.
 \bibitem{T1}[TTT] Michael Barr and Charles Wells, TOPOSES, TRIPLES AND THEORIES. Reprint in Theory and Applications of Categories, No. 12, 2005, pp. 1-288.
\end{thebibliography}
\end{document}